\documentclass{article}

\usepackage{amssymb,amsmath,url,rotating,ifthen,algorithm,algorithmicx,epsfig,multirow,array,color,multicol,graphicx,float,hhline,calc,enumitem}
\usepackage{subcaption} 
\usepackage{bold-extra,ulem}

\usepackage{algpseudocode}% http://ctan.org/pkg/algorithmicx
\algnewcommand\algorithmicinput{\textbf{Input:}}
\algnewcommand\INPUT{\item[\algorithmicinput]}
\algnewcommand\algorithmicoutput{\textbf{Output:}}
\algnewcommand\OUTPUT{\item[\algorithmicoutput]}

\algrenewcommand{\algorithmicforall}{\textbf{for each}}

\def\E{{\mathbb E}}

\def\G{\mathcal{G}}
\def\e{\epsilon }
\def\chi{{\mathbf 1}}

\def\P{{\cal P}}
\def\P{{\mathbb P}}

\def\R{{\mathbb R}}

\def\hx{\hat x}
\def\hA{\hat A}
\def\hB{\hat B}

\def\Cop{\Call{Cop}}
\def\Cops{\textsc{Cops}}
\def\Copquad{\textsc{CopQuad}}
\def\fitness{F_{\text{noisy}}}

\newtheorem{theorem}{Theorem} 
\newtheorem{remark}{Remark}
\newtheorem{proof}{Proof}

\newenvironment{keywords}{ \list{}{\advance\topsep by0.35cm\relax\small \leftmargin=1cm \labelwidth=0.35cm \listparindent=0.35cm \itemindent\listparindent \rightmargin\leftmargin}\item[\hskip\labelsep \bfseries Keywords:]} {\endlist}

\usepackage[affil-it]{authblk}

\begin{document}

%for GECCO
%\CopyrightYear{2016}
%\setcopyright{acmcopyright}
%\conferenceinfo{GECCO '16,}{July 20-24, 2016, Denver, CO, USA}
%\isbn{978-1-4503-4206-3/16/07}\acmPrice{\$15.00}
%\doi{http://dx.doi.org/10.1145/2908812.2908881}
%
%\clubpenalty=10000
%\widowpenalty = 10000

\title{Noisy Optimization: Fast Convergence Rates with Comparison-Based Algorithms}
%	\institute{ }
%\author{\inst{1}, \inst{1}}
%\numberofauthors{2}	

\author{
%\alignauthor
Marie-Liesse Cauwet
       %\affaddr{TAO, Inria, Lri, Umr Cnrs 8623}\\
       %\affaddr{Bat. 650, Univ. Paris-Sud}\\
       %\affaddr{91405 Orsay Cedex, France}\\
       %\email{marie-liesse.cauwet@inria.fr}
% 2nd. author
%\alignauthor
Olivier Teytaud}
\affil{TAO, Inria, Lri, Umr Cnrs 8623}
 %      \affaddr{TAO, Inria, Lri, Umr Cnrs 8623}\\
 %     \affaddr{Bat. 650, Univ. Paris-Sud}\\
  %     \affaddr{91405 Orsay Cedex, France}\\
   %    \email{olivier.teytaud@inria.fr}           

\date{}

\maketitle

\begin{abstract}
Derivative Free Optimization is known to be an efficient and robust method to tackle the black-box optimization problem. When it comes to noisy functions, classical comparison-based algorithms are slower than gradient-based algorithms. For quadratic functions, Evolutionary Algorithms without large mutations have a simple regret at best $O(1/ \sqrt{N})$ when $N$ is the number of function evaluations, whereas stochastic gradient descent can reach (tightly) a simple regret in $O(1/N)$. It has been conjectured that gradient approximation by finite differences (hence, not a comparison-based method) is necessary for reaching such a $O(1/N)$. We answer this conjecture in the negative, providing a comparison-based algorithm as good as gradient methods, i.e. reaching $O(1/N)$ - under the condition, however, that the noise is Gaussian. Experimental results confirm the $O(1/N)$ simple regret, i.e., squared rate compared to many published results at $O(1/\sqrt{N})$. 
\end{abstract}


\section{The black-box noisy optimization problem}
%The goal is to find an approximation of $x^*$, where $x^*$ is such that $\forall x,F(x)\geq F(x^*)$.

In a real world optimization problem, the analytical form of the objective function is frequently unavailable. It is common in this field to obtain only the fitness values of the objective function: this is the black-box problem. In this setting, given a search point, an oracle returns the corresponding fitness value. Furthermore, due to stochastic effects or inaccurate measurements, the fitness values can be improper: this is called \textit{noise}, and the optimization problem is then a noisy optimization problem. We here consider noisy optimization with constant additive Gaussian noise. Given an objective function $F$ and a search point $x\in \R^d$, the oracle outputs the fitness value $\fitness(x)$:
\begin{equation}\label{objfun}
\fitness(x)=\G(F(x),b), 
\end{equation}
where $\G(a,b)$ is a Gaussian random variable with mean $a$ and standard deviation $b>0$.

%A comparison-based A refinement of the black-box problem consists in 

%We focus on the black-box setting: the algorithm can sample $N$ times the objective function, in the oracle framework, and then should output an answer, which is an approximation $\hat{x}_N$ of the optimum $x^{*}$.
Regarding some industrial applications, a call to the oracle might be expensive, requiring heavy computations. Thus, we aim to find an approximation of the optimum within a number of evaluations as small as possible. The algorithm spends $N$ evaluations and then outputs an answer, which is an approximation - denoted $\hat{x}_N$ - of the minimum\footnote{w.l.g. we assume that the optimum is a minimum.} $x^{*}$ of $F$.
%The Simple Regret $SR$ is the classical criterion of quality for a noisy optimization algorithm.
With these notations, the simple regret after $N$ evaluations is defined by: 
\begin{equation}\label{eq:SR}
SR_{N} = \E \left( \fitness(\hat{x}_N)-\fitness(x^{*})\right) =\E F(\hat{x}_N) - F(x^*).
\end{equation}
{On the right-hand side of Eq.~\ref{eq:SR}, the expectation operates on $\hat{x}_N$ which might be a random variable due to the stochasticity of the noisy evaluations or the possible internal randomization of the optimization algorithm.}

Dupa\v{c}~\cite{dupacEnglish} has shown that noisy quadratic strongly convex functions can be optimized with simple regret $O(1/N)$, when the budget (i.e. the number of evaluations) is $N$. Fabian \cite{fabian} has broadened this result to a wider class of functions, but with only an approximation of this rate: for a function with arbitrarily many derivatives, a regret $O(1/N^\alpha)$ can be reached for $\alpha<1$ arbitrarily close to $1$. Furthermore, this bound $O(1/N)$ is optimal (see \cite{chen2}). Shamir in \cite{shamir} has improved the results, in terms of the non-asymptotic nature of some of these convergence, and in terms of explicit dependency in the dimension.

These rates are reached by algorithms introduced by Kiefer and Wolfowitz~\cite{kiefer1952stochastic}, which approximate the gradient using finite differences and thus using fitness values. However, as a refinement of the black-box problem, we might encounter some optimization problems where the fitness value itself is unknown. {\color{black}{In this case, an oracle only provides a ranking of a given set of points, but not the fitness values of these points.}} %Given two points, an oracle can only compare them and return which one is the `best' (i.e. which one has a minimum fitness value, but not the fitness value). 
For example in games, an operator can compare two agents, but not directly provide a level evaluation. In design, with the human in the loop, a user preference is a comparison between two search points. Searching a Pareto front might also involve a user providing his preferences. Comparison based algorithms such as Evolutions Strategies (ES), Differential Evolution (DE) or Particle Swarm Optimization (PSO) can handle this type of problem. The comparison oracle is also noisy in the sense that {\color{black}{the points might be misranked.}}%the point seen as the best might be indeed the worst of the two points.

Shamir in \cite{shamir} has conjectured that the use of approximate gradients is necessary for fast rates (i.e. rates $O(1/N)$) in the noisy { strongly convex} quadratic case. In this case, the best {\color{black}{known}} bounds for comparison-based algorithms are a simple regret $O(1/\sqrt{N})$ (see \cite{esareslow} for Evolution Strategies), which supports this conjecture. However, we show in the present paper that, for noisy quadratic forms, a simple regret $O(1/N)$ can be reached by a comparison-based algorithm, combining the ``mutate large inherit small'' principle \cite{BeyerMutate} and the use of large population sizes. The ``mutate large inherit small'' principle is used in the sense that we have long distances between current estimates of the optimum and search points, even when the estimate is close to the optimum. 

Jamieson \textit{et al.} in \cite{recht} have presented a bound for a comparison-based operator, using a number of comparisons quadratic $O\left(\frac{1}{\epsilon^2}\right)$ for ensuring precision $\e$ in the simple regret - whereas we only need $O\left(\frac1\epsilon \log\frac1\epsilon\right)$ comparisons. More precisely, we fully rank $O\left(\frac1\e\right)$ points; they can be sorted with $O\left(\frac1\epsilon\log\frac1\epsilon\right)$ comparisons.

Section~\ref{sec:algo} describes the key idea to get a fast comparison-based algorithm in a noisy setting. The theoretical aspects and a precise description of a fast optimization algorithm is given in Section~\ref{sec:sphere} for the specific case of the sphere function. In this case, the technicality in the proof is lighter and allowed a good insight of what we will use when switching to a larger family of functions: the quadratic functions in Section~\ref{sec:quadra}. Last, we address the experimental aspects in Section~\ref{sec:exp}.

\section{Comparison Procedure}\label{sec:algo}

%The present paper proposes an approach for fast comparison-based noisy optimization.
The main idea is to estimate the parameters of the objective function. The algorithm hence builds a model of the function and provides an approximation of the optimum. Specifically, comparing $2$ search points $N$ times provides an estimation at distance $O(1/\sqrt{N})$ of one parameter of the function.
This estimation is made possible through the frequency at which the fitness values of one of the search points is better than the other. In particular, it is crucial to know the model of noise. Hence, the optimization algorithms of Sections~\ref{sec:sphere} and~\ref{sec:quadra} consist in a sequence of calls to $\Cop{}$, given below.

\begin{algorithm}[H]
	\begin{algorithmic}[]

\Procedure{Cop}{$N$, $x$, $y$, ${\fitness}$}
	\State $f \gets 0$
	\For{$i=1$ to $N$}
	\State $f_x^i \gets {\fitness}(x)$
	\State $f_y^i \gets {\fitness}(y)$
	\EndFor	
	\State $f\gets \frac{1}{N^2} \underset{1\leq i,j \leq N}{\sum} \chi_{f_x^i < f_y^j}$
	
	\Return $f$
\EndProcedure
	\end{algorithmic}	
	\caption*{{\bf Comparison Procedure ($\Cop{}$)}.}
\end{algorithm}
{\color{black}{
Importantly, this operator can be computed faster than the apparent $O(N^2)$ complexity. Using sorting algorithm, the complexity is $O(N\log N)$. %We provide in $\widehat{\Cop{}}()$ a rewriting of $\Cop$ enlightening this complexity.

%\begin{algorithm}[H]
%	\begin{algorithmic}[]
%
%\Procedure{$\widehat{\text{\rm{Cop}}}$}{$N$, $x$, $y$, ${\fitness}$}
%	\State $f \gets 0$
%	\State Evaluation $N$ fitness evaluations at $x$, $f_x^i=\fitness(x)$
%		and $f_y^i=fitness(y)$ for $i\in \{1,\dots,N\}$.
%	\State Let $idx(f_{x}^i)$ and $idx(f_y^i)$ be the index of $f_{x}^i$ and $f_{y}^i$ in a full ranking (randomly break ties, if any), as in Eqs. \ref{bif}-\ref{bof} (in case of no tie):
%\begin{eqnarray}
%	idx(f_x^i)\in \{1,\dots,2N\}& & \label{bif}\\
%	idx(f_y^i)\in \{1,\dots,2N\}& & \\
%	\left(idx(f_x^i)< idx(f_x^j)\right)&\Longleftrightarrow& \left( f_x^i < f_x^j \right)\\
%	\left(idx(f_x^i)< idx(f_y^j)\right)&\Longleftrightarrow& \left( f_x^i < f_y^j \right)\\
%	\left(idx(f_y^i)< idx(f_x^j)\right)&\Longleftrightarrow& \left( f_y^i < f_x^j \right)\\
%	\left(idx(f_y^i)< idx(f_y^j)\right)&\Longleftrightarrow& \left( f_y^i < f_y^j \right)\label{bof}
%\end{eqnarray}
%	\State Rank $N$ noisy fitness values of $\fitness$ at $x$ and $N$ noisy fitness values of $\fitness$ at $y$:
%	$$f_{\ast}^1 \leq \dots \leq f_{\ast}^{2N}$$ 
	%\For{$i=1$ to $N$}
	%\State $f_x^i \gets {\fitness}(x)$
	%\State $f_y^i \gets {\fitness}(y)$
	%\EndFor	
	%\State $f\gets \frac{1}{N^2} \underset{1\leq i,j \leq N}{\sum} \chi_{f_x^i < f_y^j}$
%	\For{$i=1$ to $N$}
%		\State{$f\gets idx(f_{y}^{i})-i$ }
%	\EndFor
%	\State $f=f/N^2$
%	\Return $f$
%\EndProcedure
%	\end{algorithmic}	
%	\caption*{{\bf Comparison Procedure ($\widehat{\Cop{}}$)}. $\ast$ stands for $x$ or $y$, $idx$ refers to the component index.}
%\end{algorithm}
}}
%We first present such an algorithm when optimizing a sphere function in Section~\ref{sec:sphere}, then we switch to a larger family of functions, namely the quadratic forms, in Section~\ref{sec:quadra}.  

%3) Remarks: not online algo.   
\vfill\break
\section{Sphere function}\label{sec:sphere}

\subsection{In dimension $1$}

We first propose in Alg. \ref{algo:cops1} an algorithm ($\Cops 1$) achieving regret $O(1/N)$ on the noisy sphere problem in dimension $1$.

\begin{algorithm}[H]
\begin{algorithmic}[]
\INPUT
\Statex an oracle ${\fitness}: x\in \mathbb{R} \mapsto \G(|x-x^*|^2,1)$ 
\Statex an even budget $N$
\OUTPUT
\Statex an approximation $\hat x$ of the optimum $x^*\in [-1,1]$ of the objective function ${F:\ x\mapsto |x-x^*|^2}$

\noindent \Statex  \hrulefill
\Statex $K \gets N/2$
\Statex $f \gets \Cop{K,1,-1, {\fitness}}$ 

\Statex Define $\hat x$ such that $\P\left(\G(0,1)< \sqrt{8} \hat{x}\right)=f$
\Statex $\hat x\leftarrow \max(-1,\min(1,\hat x))$

\Return $\hat{x}$ 
\caption{{ Comparison Procedure for Sphere function in dimension $1$ ($\Cops 1$).}\label{algo:cops1}}
\end{algorithmic}
\end{algorithm}

\begin{theorem}\label{peredesths}
Let $\fitness(x)=|x-x^{*}|^2+\G(0,1)$ be the noisy sphere function in dimension $1$, where ${x^*\in [-1,1]}$. Then the simple regret of $\Cops 1$ after $N$ evaluations satisfies:
\begin{equation}
SR_N = O(1/N).
\end{equation}
\end{theorem}

\begin{proof}
Consider $\Cops 1$ on such an objective function. By definition of $\fitness$ and $F$,
%{\bf{Step 1: Probability $p$ of ${\fitness}(1)<{\fitness}(-1)$.}}
{\color{black}
\begin{align}
p&=\P\left({\fitness}(1)<{\fitness}(-1)\right)\nonumber\\
 &=\P\left(|1-x^*|^2+\G(0,1)<|-1-x^*|^2+\G(0,1)\right)\nonumber\\
 &=\P\left(\sqrt{2}\G(0,1)<(1+x^*)^2-(1-x^*)^2\right)\nonumber\\
 &=\P\left(\G(0,1)<\sqrt{8}x^*\right)\label{step1}.
\end{align}
}
%The probability of ${\fitness}(1)<{\fitness}(-1)$ is the probability $p$
%that $\G(F(1),1)<\G(F(-1),1)$.
%This is  also $p=\P(\G(F(1)-F(-1),\sqrt{2})<0)$.
%$F(1)=|x^*-1|^2$ and $F(-1)=|1+x^*|^2$, hence
%$p=\P(\G(0,\sqrt{2})<4x^{*})$, or in other words 
%\begin{equation}
%p=\P(\G(0,1)<\sqrt{8}x^*).
%\end{equation}

{\bf{Step 1: Expectation and Variance of $f$.}}

With the notations of $\Cop{}$, let us define:
\begin{eqnarray*}
\forall\ i, j\in \{1,\dots,N\}^2,\ \chi_{i,j}  =\begin{cases}               																1\ \text{if}\ f_1^i<f_{-1}^j\\
              									0\ \text{otherwise}\\
            										\end{cases}
\end{eqnarray*}

$\chi_{i,j}$ is Bernoulli distributed with probability of success $p$.

$f$ is the output of the \Cop{}~procedure. By definition,
\begin{equation*}
f=\frac{1}{K^2}\underset{1\leq i,j \leq K}{\sum}\chi_{i,j}.
\end{equation*}

The expectation and variance of $f$ are then: 
\begin{align}
\E f &= p\nonumber\\
Var f &= \frac{1}{K^4} \sum_{i=1}^{K} \sum_{j=1}^{K} Cov \left(\sum_{k=1}^{K}\chi_{i,k}, \sum_{k'=1}^{K}\chi_{j,k'}\right)\nonumber\\
 &= \frac{1}{K^4}\sum_{i=1}^{K} \sum_{j=1}^{K}\sum_{k=1}^{K}\sum_{k'=1}^{K} Cov (\chi_{i,k}, \chi_{j,k'})\label{eq:var}
\end{align} 
If $i\neq j$ and $k\neq k'$, $Cov (\chi_{i,k}, \chi_{j,k'})=0$ by independence. If $i=j$ (or $k=k'$), by Cauchy-Schwarz:

\begin{equation*}
Cov (\chi_{i,k}, \chi_{i,k'})\leq \sqrt{Var(\chi_{i,k})Var(\chi_{i,k'})}\leq \frac{1}{4}
\end{equation*}
This together with Eq.~\ref{eq:var} give:

\begin{align*}
Var f &= \frac{1}{K^4}\left(\sum_{i=1}^{K}\sum_{k=1}^{K}\sum_{k'=1}^{K} Cov (\chi_{i,k}, \chi_{i,k'})+\right. \\
&\ \ \ \ \ \ \ \ \left.\sum_{i=1}^{K}\sum_{j=1}^{K}\sum_{k=1}^{K} Cov (\chi_{i,k}, \chi_{j,k})\right),\\
	&\leq\frac{1}{K^4}\times \frac{K^3}{2},\\
	&\leq \frac{1}{N}. 
\end{align*}

{\bf{Step 2: Lipschitz.}}
We denote by $\Phi$ the cumulative distribution function of the standard Gaussian: ${\Phi(x)=\P \left(\G(0,1)<  x \right)}$ and $m$ and $M$ such that ${\Phi_{[m,M]}^{-1}: [m,M]\rightarrow [-1,1]}$ is the inverse of $\Phi$ over these intervals. Let us define  

\begin{eqnarray*}
 h(x)=\begin{cases}
	\Phi_{[m,M]}^{-1}(x)\ &\text{if}\ m\leq x \leq M\\
	-1\ &\text{if}\ x<m\\
     1\ &\text{if}\ M<x\\
     \end{cases}
\end{eqnarray*} 

Let us evaluate the Lipschitz coefficient $L(h)$ of $h$.  $\Phi_{[m,M]}^{-1}$ is differentiable over $[m,M]$ since $\Phi$ is differentiable over $[-1,1]$ hence its Lipschitz $L(\Phi_{[m,M]}^{-1})$ is bounded. $h$ is continuous, and $h$ is constant over $(-\infty,m]$ and $[M,\infty)$; hence the Lipschitz of $h$ is $L(\Phi_{[m,M]}^{-1})$ over $[m,M]$.
%$y$ is differentiable over $[m,M]$, hence its Lipschitz is bounded.

%Let us define $h(x)=\min(1,\max(-1,y(x)))$,
%where $y(x)$ is such that $P(G(0,1)\geq \sqrt{8}y(x))=x$.
%Let us evaluate the Lipstchitz coefficient $L(h)$ of $h$.
%
%Define $m$ such that $y(m)=1$ and $M$ such that $y(M)=-1$.
%$y$ is differentiable over $[m,M]$, hence its Lipschitz is bounded.
%
%$h$ is continuous, and $h$ is constant over $(-\infty,m]$ and $[M,\infty)$; hence the Lipschitz of $h$ is its Lipschitz over $[m,M]$, where it is equal to $y$.

{\bf{Step 3: Concluding.}}
We have, by definition of $\Cops 1$ for $\hat x$ and by Eq.~\ref{step1} for $x^*$, 
\begin{equation}
\hat x = \frac{h(f)}{\sqrt{8}}\mbox{ and }x^*=\frac{h(p)}{\sqrt{8}},\label{transp}
\end{equation}
By definition of the simple regret in Eq.~\ref{eq:SR}, 
\begin{align*}
SR_{N}	&=\E |\hat{x}-x^*|^{2}\\
		& \leq \E L(h)^{2}|f-p|^{2}/8 \mbox{ by {\bf Step 2}}\\
		& \leq \frac{L(h)^{2} }{8N} \mbox{ by {\bf Step 1}}.
\end{align*}
%\CQFD
%By step 2, 
%\begin{equation}
%E (f-p)^2\leq \frac1{2N}\label{tota}.
%\end{equation}
%By step 3 and Eq. \ref{transp}, 
% $|\hat x-x^*|\leq L(h)|f-p|$, where $L(h)$ is the Lipschitz coefficient of $h$ over $[0,1]$.
%
%Hence,
%$E (\hat x- x^*)^2 \leq L(h)^2E(f-p)^2$
%and therefore, by Eq. \ref{tota}, $E(\hat x- x^*)^2\leq L(h)^2/(2N)$.
\end{proof}

\begin{remark}
The result of Theorem~\ref{peredesths} is based on the fact that the noise is a standard Gaussian. However, this result still holds as soon as the noise distribution has expectation 0, finite variance (possibly unknown, see Section~\ref{sec:quadra}) and a bounded Lipschitz.   The distribution of the noise, on the other hand, must be known.
\end{remark}

\subsection{Multidimensional sphere function}

Alg.~\ref{alg:cops} ($\Cops$) presents a straightforward extension to the noisy multidimensional sphere. $B_d(c,r)$ denotes the ball of center $c$ and radius $r$ in dimension $d$, and $\|.\|$ is the Euclidean norm. %$B_d(c,r)=\{x\in \R^d; \|x-c\|\leq r\}$, where $\|.\|$ denotes the Euclidean norm in $\R^d$. % {\color{blue}When the dimension is implicit, we drop the index $d$, and $B(c,r)=B_d(c,r)$.}

\begin{algorithm}
\begin{algorithmic}
\INPUT
\Statex an oracle ${\fitness}: x\in \mathbb{R}^d \mapsto \G(\|x-x^*\|^2,1)$ 
\Statex a budget $N$ (multiple of $2d$)
\OUTPUT
\Statex an approximation $\hat x$ of the optimum $x^*\in B_d(0,1)$ of the objective function ${F:\ x\mapsto \|x-x^*\|^2}$
\noindent \Statex  \hrulefill
\Statex $K\gets N/2d$
\For{$i=1$ to $d$}
\Statex Apply $\Cops1$ with a budget $K$ on the unidimensional restriction of $\fitness$ to $\{0\}^{i-1}\times [-1,1]\times \{0\}^{d-i}$
\Statex $\hat{x}_i$ be the obtained approximation of the optimum in $[-1,1]$.
\EndFor

\Return $\hat{x}=(\hat{x}_1,\dots,\hat{x}_d)$.
\end{algorithmic}
\caption{Comparison procedure for the sphere function (\Cops).\label{alg:cops}}
\end{algorithm}

\begin{theorem}
Let $\fitness(x)=\|x-x^{*}\|^{2}+\G(0,1)$ be the noisy sphere function, with $x^{*}\in B_{d}(0,1)\subset \R^{d}$.Then the simple regret of \Cops~after $N$ evaluations is:
\begin{equation*}
SR_{N}=O(d/N).
\end{equation*}
%$\Cops$ has regret $O(d/N)$ over objective functions as in Eq. \ref{objfun} when $F(x)=\|x-x^*\|^2$ and $x^*\in B_d(0,1)\subset \R^d$.
\end{theorem}

\begin{proof}
The conditions of Theorem 1 are verified for each application of $\Cops1$. The simple regret for the multidimensional case is the sum of the simple regrets of each restrictions.
\end{proof}

\section{General quadratic forms}\label{sec:quadra}

Alg.~\ref{alg:copquad} extends the principle of Section~\ref{sec:sphere} to the optimization of a wider class of quadratic functions. $\|\cdot\|_{2}$ denotes the matrix norm induced by $\|\cdot\|$, i.e. ${\|A\|_{2}=\underset{x\neq 0}{\sup} \frac{\|Ax\|}{\|x\|}}$ and $\|\cdot\|_{F}$ is the Frobenius norm. $(e_i)$ is the standard basis and $A^{t}$ is the transpose of matrix $A$.

\begin{algorithm}
\begin{algorithmic}[1]
\scriptsize
\INPUT
\Statex an oracle ${\fitness}: x\in \mathbb{R}^d \mapsto \G(x^{t} A x +B x +C,D)$ 
\Statex a budget $N$ (multiple of $d(d+3)-2$)
\OUTPUT
\Statex an approximation $\hat x$ of the optimum $x^*\in B_d(0,1)$ of the objective function ${F:\ x\mapsto x^{t} A x +B x +C}$
\noindent \Statex  \hrulefill
\State $K\gets\frac{N}{d(d+3)-2} $

\For{$i=1$ to $d$}
\State $f_{-e_i,e_i}\gets\Cop{K,- e_i, e_i, \fitness }$
\State Define $\hat{B}_{i}(D)$ such that $\P\left(\G(0,1)< \sqrt{2}\hat{B}_i(D) \right)=f_{-e_i,e_i}$
\State $\hat{B}_{i}(D)\gets\max(-5,\min(\hat{B}_{i}(D),5))$\Comment{\small Estimate of $B_i/D$}

\State $f_{0,e_i}\gets\Cop{K,0 , e_i, \fitness }$
\State Define $\theta_{ii}(D)$ such that $\P\left( \G(0,1)<  \theta_{ii}(D)/\sqrt{2}\right)=f_{0,e_i}$
\State $\theta_{ii}(D)\gets\max(-5,\min(\theta_{ii}(D),5))$
\State $\hat{A}_{i,i}(D)\gets\theta_{ii}(D)-\hat{B}_i(D)$\Comment{\small Estimate of $A_{i,i}/D$}
\EndFor

\For{$i=1$ to $d$}
\For{$j=i+1$ to $d$}
\State $f_{0,e_i +e_j}\gets\Cop{K, 0, e_i+e_j, \fitness }$
\State Define $\theta_{ij}(D)$ such that $${\P( \G(0,1)<  \theta_{ij}(D)/\sqrt{2})=f_{0,e_i+e_j}}$$
\State $\theta_{ij}(D)\gets\max(-5,\min(\theta_{ij}(D),5))$
\State $\hat{A}_{i,j}(D)\gets\frac{1}{2} (\theta_{ij}(D)-\hat{B}_i(D)$\\
$\hspace*{4cm}-\hat{A}_{i,i}(D)-\hat{B}_j(D)-\hat{A}_{j,j}(D))$
\State $\hat{A}_{j,i}(D)\gets\hat{A}_{i,j}(D)$\Comment{\small Estimate of $A_{i,j}/D$ and $A_{j,i}/D$}
\EndFor
\EndFor

\State $\hat{A}(D)\gets (\hat{A}_{i,j}(D))$
\State $\hat{B}(D)\gets (\hat{B}_{i}(D))$

\If{$\hat{A}(D)$ is not singular}\label{line:proj}
\State $\hat{x}\gets -\frac12\hat{B}(D)^{t}\hat{A}(D)^{-1}$
\Else
\State{$\hat{x}\gets 0$}
\EndIf

\Return $\hat{x}\gets $ projection of $\hat x$ on $B_d(0,1)$.
\end{algorithmic}
\caption{Comparison procedure for quadratic functions ($\Copquad$).}\label{alg:copquad}
\end{algorithm}

\begin{theorem}\label{poulala}
Let $\e\in ]0,1[$. Consider an objective function ${\fitness(x)=x^{t}Ax+Bx+C+D\G(0,1)}$, with optimum $x^*$ in $B_d(0,1-\e)\subset \R^d$, and $D>0$. Assume that $\frac1D \|B\|\leq 1$ and $\frac1D|C|\leq 1$. If $A$ is symmetric positive definite such that its eigenvalues are lower bounded by some $c>0$ and $\|\frac1D A\|_{2} \leq 1$, 
then, when applying \Copquad, 
{\color{black}{${SR_N=O(\max((\lambda_{\max}/\lambda_{\min})^2,\lambda_{max}^2) D^2/N)}$}}, 
where $\lambda_{\max}$ is the maximum eigenvalue of $\frac1DA$, and $\lambda_{min}>\frac1D c$ is the minimum eigenvalue.
\end{theorem}

{\bf{Remark:}} Please note that $\lambda_{max}\leq 1$ by the assumptions in Theorem \ref{poulala}.

\begin{proof}
%We first recall a classical result on the adjusted estimator $\hat D$ (see Line \ref{line:var} in Alg. \ref{alg:copquad}) of the standard deviation $D$ (see for example~\cite{taboga}):   
%\begin{equation}\label{vard}
%\E (\hat D-D)^2=O(D^2/N).
%\end{equation}

Let $x$ and $y$ be two points to be compared in \Copquad: ${(x,y)\in\mathcal{C}:=\{(e_i,-e_i)_i, (0,e_i)_i, (0,e_i+e_j)_{i\neq j}\}}$. We denote by $\Delta_{x,y}$ the value ${\Delta_{x,y}:=\E (\fitness(y)-\fitness(x))=F(y)-F(x)}$ and by $f_{x,y}$ the frequency $f_{x,y}:=\frac{1}{K^2}\sum_{1 \leq i,j \leq K}\chi_{f_x^i<f_y^j}$, where $f_x^i$ and $f_y^j$ are as in Section~\ref{sec:algo}.

{\bf{Step 1: Mean Squared Error of frequencies.}} 

%at which $f_i^x<f_j^y$ over $N$ independent realizations $f_i^x$ (resp. $f_j^y$) of $\fitness(x)$ (resp. $\fitness(y)$):  Particularly, with these notations, $ \P\left(\fitness(x) < \fitness(y)\right)=\P\left(\G(0,1) < \frac{\Delta_{x,y}}{\sqrt{2}D}\right)=g\left(\frac{\Delta_{x,y}}{\sqrt{2}D}\right).$

Similarly to step 2 of Theorem \ref{peredesths}, and using the notation ${\Phi(x)=\P(\G(0,1)<x)}$, 
\begin{align}
\E(f_{x,y})&=\Phi\left(\frac{\Delta_{x,y}}{\sqrt{2}D}\right)\nonumber\\
\E\left(f_{x,y}-\Phi\left(\frac{\Delta_{x,y}}{\sqrt{2}D}\right)\right)^2&=Var(f_{x,y})=O(1/N).\label{eq:mse}
\end{align}
%{\bf{Step 1: we can estimate differences with frequencies.}} With $\Delta_{x,y}=\E f(y)-f(x)$, $f_{x,y}$ the frequency at which $f_i^x<f_j^y$ over $N$ independent realizations $f_i^x$ (resp. $f_j^y$) of $f(x)$ (resp. $f(y)$), 
%$$\E(G^{-1}(\Delta_{x,y}/D)-f_{x,y})^2=O(1/N).$$
%
%This proof is the same as the proof of steps 1 and 2 in Theorem \ref{peredesths}. 

{\bf Step 2: Mean Squared Error of ${\hat{A}(D)}$ and ${\hat{B}(D)}$.} 

As in Step 3 of the proof of theorem~\ref{peredesths}, we denote by $\Phi_{[\tilde{m},\tilde{M}]}^{-1}: [\tilde{m},\tilde{M}] \rightarrow [-5,5]$ the inverse of $\Phi$ over these intervals: 

\begin{eqnarray*}
 \tilde{h}(x)=\begin{cases}
	\Phi_{[\tilde{m},\tilde{M}]}^{-1}(x)\ &\text{if}\ \tilde{m}\leq x \leq \tilde{M}\\
	-5\ &\text{if}\ x<\tilde{m}\\
     5\ &\text{if}\ \tilde{M}<x\\
     \end{cases}
\end{eqnarray*}

By assumption, $(x,y)\in \mathcal{C}$, $\frac1D\|A\|_{2}\leq 1$ and $\frac1D\|B\|\leq 1$, $\Delta_{x,y}/\sqrt{2} D \in [-5,5]$ and then, as in Step 3 and 4 of Theorem  \ref{peredesths},

\begin{align}
\E\left(\tilde{h}(f_{x,y})-\frac{\Delta_{x,y}}{\sqrt{2}D}\right)^2&\leq \E\left(\tilde{h}(f_{x,y})-\tilde{h}\left(\Phi\left(\frac{\Delta_{x,y}}{\sqrt{2}D}\right)\right)\right)^2\nonumber\\
%\E(h(f_{x,y})-\Delta_{x,y}/(\sqrt{2}D))^2
	&\leq L(\tilde{h})^2\E\left(f_{x,y}-\Phi\left(\frac{\Delta_{x,y}}{\sqrt{2}D}\right)\right)^2\nonumber\\
&=O(1/N)\mbox{ by Eq.~\ref{eq:mse}.}\label{eq:bigo}
\end{align}

By applying Eq.~\ref{eq:bigo}, we then estimate the mean squared error of $\hat{A}(D)$ and $\hat{B}(D)$:
\begin{itemize}
\item ${\hat{B_i}(D)=\sqrt{2}\tilde{h}(f_{-e_i,e_i})/2}$ and ${B_i/D=\Delta_{-e_i,e_i}/2D}$ ${\forall i\in\{1,\dots,d\}}$, then ${\E(\hat{B_i}(D)-B_{i}/D)^{2}=O(1/N)}$ by Eq.~\ref{eq:bigo}, hence ${\E\|\hat{B}(D)-B/D\|^{2}=O(1/N)}$.

\item ${\hat{A}_{i,i}(D)=\sqrt{2}\tilde{h}(f_{0,e_i})-\hat{B}_{i}(D)}$ and ${A_{i,i}/D=\Delta_{0,e_i}/D-B_i/D}$, then $\E(\hat{A}_{i,i}(D)-A_{i,i}/D)^{2}=O(1/N)$ using Eq.~\ref{eq:bigo},
and $$\E(\hat{B_i}(D)-B_{i}/D)^{2}=O(1/N).$$ If $i\neq j$, then $$\hat{A}_{i,j}(D)=\frac{1}{2}\left(\sqrt{2}\tilde{h}(f_{0,e_i+e_j})\right.\ \ \ \ \ \ \ \ \ \ \ \ \ \ $$ $$\left.-\hat{B}_{i}(D)-\hat{A}_{i,i}(D)-\hat{B}_{j}(D)-\hat{A}_{j,j}(D)\right),$$ 
and $$A_{i,j}/D=\ \ \ \ $$
$$1/2\left(\Delta_{0,e_i+e_j}/D-B_i/D-A_{i,i}/D-B_j/D-A_{j,j}/D\right)$$ hence, by proceeding as above, $$\E(\hat{A}_{i,j}(D)-A_{i,j}/D)^{2}=O(1/N)$$ and $$\E\|\hat{A}(D)-A/D\|_{F}^{2}=O(1/N).$$
\end{itemize}
%${A_{i,i}/D=\Delta_{0,e_i}/D-B_i/D}$, then $\E(\hat{A}_{i,i}(D)-A_{i,i}/D)^{2} =\E(\sqrt{2}\tilde{h}(f_{0,e_i})-\hat{B}_{i}(D)-(\Delta_{0,e_i}/D-B_{i}/D))^{2}=O(1/N)$
%For $(x,y)=(e_i,e_i),\ (0,e_i)\ \text{or}\ (0,e_i+e_j)$
%\hat{A}_i&=\sqrt{2}\hat{D}h(f_{0,e_i}) -\hat{B}_{i}\ \text{ and}\\
%\hat{A}_{i,j}&=\frac{1}{2}\left(\sqrt{2}\hat{D}h(f_{0,e_i+e_j})-\hat{B}_{i}-\hat{A}_{i,i}-\hat{B}_{j}-\hat{A}_{j,j}\right)
%For $(x,y)=(e_i,e_i),\ (0,e_i)\ \text{or}\ (0,e_i+e_j)$,
%{\bf{Step 2: the algorithm generates $\hat A$ and $\hat B$ such that %$\hat A$ has the size of $A$, $\hat B$ has the size of $B$, and 
%the expected squared Frobenius norm of $\hat A-A$ (resp. $\hat B-B$) is $O(D^2/N)$.}}
%
%Step 2 is proved by the following equations:
%\begin{eqnarray*}
%b_i&=&\frac12 \E (f(e_i)-f(-e_i))\\
%a_{i,i}&=&\E (f(e_i)-f(0))-b_i\\
%2a_{i,j}&=&\E (f(e_i+e_j)-f(0))-a_{i,i}-a_{j,j}
%\end{eqnarray*}
%Each difference of $f(x)-f(y)$ can be evaluated with the requested precision as explained in Step 1. 
% Therefore we get $\hat A$ and $\hat B$, approximations of $A$ and $B$ respectively, with squared error $O(D^2/N)$.

{\bf{Step 3: with probability at least $1-O(1/N)$, \Copquad~returns an estimate $\hat{x}$ solution of ${2\hx\hA(D)=-\hB^t(D)}$.}}

By definition of \Copquad~, $2\hx\hA(D)\neq -\hB^t(D)$ only if $\hx$ could not be properly defined because $\hA(D)$ is singular or if we use the projection.

The eigenvalues are continuous (see e.g. \cite{magiceigen}); therefore in a neighborhood of $A/D$, $\hA(D)$ has eigenvalues lower bounded by some $\delta>0$.
Therefore, $\hA(D)$ is singular only out of this neighborhood; this occurs, by {\color{black}{Markov}}'s inequality, with probability $O(1/N)$. Therefore, the first case occurs with probability at most $O(1/N)$.

With probability at least $1-O(1/N)$, the solution $\hx$ of $2\hx\hA(D)=-\hB^t(D)$ is therefore the projection of $-\frac12 \hB(D)^t\hA(D)^{-1}$.
For $\hA(D)$ close enough to $A/D$ and $\hB(D)$ close enough to $B/D$, 
this is close to $x^*$, and therefore it is inside $B_d(0,1-\e)$.

{\bf{Step 4: concluding when $2\hat x\hA(D)=-\hB(D)^t$.}}

Define $B'=B/D-\hB(D)$ and $A'=A/D-\hA(D)$.
We have $2x^*A=-B^t$ and $2\hat x \hA(D)=-\hB(D)^t$.

By substraction, we get 
$$2(\hx\hA(D)-x^*A/D)=(B/D)^t-\hB(D)^t$$
hence $2(\hx A/D-\hx A'-x^*A/D)=B'^t$, using definitions of $A'$ and $B'$.

By step 2, all terms in $A'$ and $B'$ have expected squared norm $O(1/N)$; and by step $3$ $\hx$ is bounded, therefore
$$2(\hx A/D-x^*A/D)=B'^t+2\hx A'$$ has expected squared norm $O(1/N)$,
and $$(\hx-x^*)=\frac12uA^{-1}D$$ with $\E \|u\|^2=O(1/N)$.

With $\lambda_{min}>0$ the smallest eigenvalue of $\frac1D A$, we get
${\E\|\hx-x^*\|^2=O(\lambda_{min}^{-2}/N)}$.

{\color{black} Note that $F$ can be rewritten as 
$${F(x)=(x-x^{*})^t A (x-x^{*}) + C'},$$ where $x^*=-\frac{1}{2}B^t A^{-1}$ and $C'=C-{x^{*}}^{t}Ax^*$.
\begin{align*}
\text{Then } SR_N&=\|F(\hx)-F(x^*)\|^2=\|(\hx-x^{*})^t A (\hx-x^{*})\|^2\\ 
&\leq \lambda_{max}^2\|\hx-x^*\|^2\\ 
\text{Hence  } SR_N&=O\left(\left(\frac{\lambda_{max}}{\lambda_{min}}\right)^2\frac{D^2}{N}\right)\text{, which is the expected }
\end{align*}
}result.

{\bf{Step 5:} General conclusion}

Let us denote by $\mathcal{S}$ the event ``\Copquad~returns an estimate $\hat{x}$ solution of ${2\hx\hA(D)=-\hB(D)^t}$'' and $\bar{\mathcal{S}}$ its complement. In the following, $diam$ denotes the diameter. By definition,
\begin{align*}
SR_{N}&=\E(\fitness(\hx)-\fitness(x^*))\\
&=\underbrace{\E(\fitness(\hx)-\fitness(x^*)|\mathcal{S})}_{=O\left(\left(\frac{\lambda_{max}}{\lambda_{min}}\right)^2\frac{D^2}{N}\right) \text{\bf by step $4$}}\underbrace{\P(\mathcal{S})}_{\leq 1}\\
& +\underbrace{\E(\fitness(\hx)-\fitness(x^*)|\bar{\mathcal{S}})}_{\leq \lambda_{\max}^2\times D^2\times diam (B_d(0,1-\e))}\underbrace{\P(\bar{\mathcal{S}})}_{=O(1/N) \text{\bf by step $3$}}
\end{align*}
%{\bf{Step 5:}} 
%Step 4 shows a squared deviation $O(D^2/N)$, when $\hx$ is well defined and does not need any projection. By step 3, this happens with probability at least $1-O(D^2/N)$. 
%
%The simple regret when Step 4 can not be applied is upper bounded by the squared diameter of $B(0,1-\e)$ multiplied by the squared maximum eigenvalue, and this happens with probability $O(D^2/N)$. 
%
%The overall squared deviation is therefore upper bounded by 
%\begin{eqnarray*}
%& & P(\mbox{Step 4 can not be applied})\times \mbox{$\lambda_{max}^2\times$ (diameter of $B(0,1-\e)$)$^2$}\\
%&+& P(\mbox{Step 4 can be applied})\times \mbox{ squared deviation with step 4.}.
%\end{eqnarray*}
Hence the expected result.
\end{proof}
%Continuity and Location of Zeroes of Linear Combinations of Polynomials, by Mishael Zedek, Proc. Amer. Math. Soc. 16 (1965), 78-84

\section{Experiments}\label{sec:exp}

For each experiment, parameters $A$, $B$ and $C$ satisfying assumptions in Theorem~\ref{poulala} are randomly generated. \Copquad~then returns an approximation of the optimum of the noisy quadratic function ${F(x)=x^{t}Ax+Bx+C+D\G(0,1)}$. Results are obtained over $50$ runs.% TODO explain log-log scale.
%Fig. \ref{mcafig} presents results of MCA in dimension 1 and of MMCA in dimension 2.
%The linear rate (in log-log scale) with slope $-1$ is clearly visible.

{\bf \Copquad~to tackle strong noise.} Fig.~\ref{fig:dim2} presents results of \Copquad~in dimension $2$ when the standard deviation $D$ satisfies the assumptions in Theorem~\ref{poulala}, i.e., $\|B\|/D\leq 1$, $|C|/D\leq 1$ and $\|A\|_{2}/D\leq 1$. The linear rate (in log-log scale) with slope $-1$ is clearly visible. We obtained similar graphs (not presented here) for dimensions $5$.

\begin{figure}
   \centering
    \begin{subfigure}[b]{0.49\textwidth}
        \includegraphics[width=\textwidth]{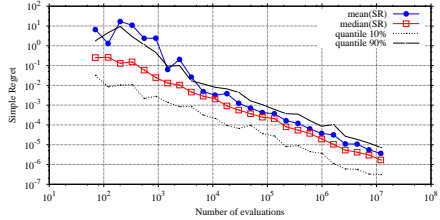}
        \caption{$D=1$}\label{fig:dim2noise1}
    \end{subfigure}
    ~ 
    \begin{subfigure}[b]{0.49\textwidth}
        \includegraphics[width=\textwidth]{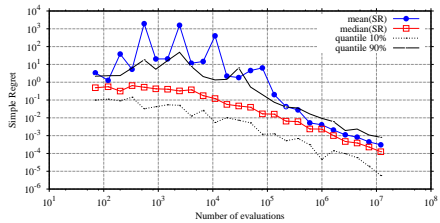}
        \caption{$D=10$}
        \label{fig:dim2noise5}
    \end{subfigure}
    \caption{Dimension $d=2$, over $50$ runs. Mean, median and quantiles $10\%$ and $90\%$ are displayed.}\label{fig:dim2}
\end{figure}

{\bf \Copquad~with small noise.} Figure~\ref{fig:smallnoise} then shows the case of a smaller noise $D$ for dimension $2$. Along with the theory ( $\|A/D\|_2$ does not satisfy the assumptions), we lose the $O(1/N)$ rate. In the early stages, \Copquad~still seems to converge, but it eventually stagnates around the optimum. It is counter-intuitive that an algorithm performs worse when noise decreases; nonetheless, in the case $\frac1D A\rightarrow 0$, the \Cop{}~operator always return $0$ or $1$, thus the estimated parameters are $-5$ or $5$, and the algorithm does not converge. Incidentally, this is consistent with the bandit literature, where the hardest cases are when optimal arms have close values.
Providing an algorithm able to cope with $D\leq\|A\|_{2}$ is possible - asymptotically, as for bandit algorithms mentioned above. Progressively widening the projection interval $[-b(N),b(N)]$ instead of keeping $[-5,5]$ fixed {\color{black}makes this possible; if we have a slow enough function $b:N\mapsto b(N)$ for defining the interval $[-b(N),b(N)]$, then we get:
\begin{itemize}
\item e.g. $\log(\log(\log(N)))$ in Eq. \ref{eq:bigo}, 
\item and asymptotically we still get a probability $1/N$ in Step 3 of Theorem \ref{poulala}.
\end{itemize}

So that, for $N>N_0$, we get Theorem \ref{poulala} (up to the slight increase in the bound, depending on the choice of the $b$ function) independently of $D\leq \|A\|_2$  - but $N_0$ depends on $\frac1DA$.}

\begin{figure}
   \centering
%    \begin{subfigure}[b]{0.45\textwidth}
        \includegraphics[width=0.49\textwidth]{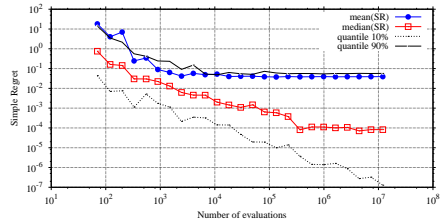}
%        \caption{$d=2$}\label{fig:dim2noise03}
%   \end{subfigure}
%    ~ 
%    \begin{subfigure}[b]{0.45\textwidth}
%        \includegraphics[width=\textwidth]{figures/mcaDim5Noise03nbRuns50.eps}
%        \caption{$d=5$}
%        \label{fig:dim5noise03}
%    \end{subfigure}  
%    ~ 
%    \begin{subfigure}[b]{0.45\textwidth}
%        \includegraphics[width=\textwidth]{figures/mcaDim10Noise03nbRuns50.eps}
%        \caption{$d=10$}
%        \label{fig:dim10noise03}
%    \end{subfigure}
    \caption{$d=2$, $D=0.65$.}\label{fig:smallnoise}
\end{figure}

\section{Conclusion}\label{sec:conc}

We have shown that comparison-based algorithms can reach a regret $O(1/N)$ on quadratic forms. This partially solves (negatively) a conjecture in \cite{shamir}, and improves results proposed in \cite{decockfoga, rolet2010adaptive}. Our main assumption is the Gaussian nature of the noise. We do not assume that the variance is known, but it is supposed to be constant.

{\bf Future work.} We assume an exactly quadratic function; maybe rates in $O(1/N^{2/3})$ can be reached for non-quadratic functions under smoothness assumptions. Also we might extend the present results to non Gaussian noise.

\bibliographystyle{abbrv}
\bibliography{mca}

\begin{thebibliography}{10}

\bibitem{esareslow}
S.~Astete-Morales, M.-L. Cauwet, and O.~Teytaud.
\newblock {Evolution Strategies with Additive Noise: A Convergence Rate Lower
  Bound}.
\newblock In {\em {Foundations of Genetic Algorithms}}, Foundations of Genetic
  Algorithms, page~9, Aberythswyth, United Kingdom, 2015.

\bibitem{BeyerMutate}
H.-G. Beyer.
\newblock {Mutate Large, But Inherit Small! On the Analysis of Rescaled
  Mutations in $(\tilde{1}, \tilde{\lambda})$-ES with Noisy Fitness Data}.
\newblock In {\em Parallel Problem Solving from Nature,~5}, Heidelberg, 1998.
  Springer.
\newblock in print.

\bibitem{chen2}
H.~F. Chen, T.~E. Duncan, and B.~Pasik-Duncan.
\newblock {A stochastic approximation algorithm with random differences}.
\newblock In {\em Proceedings of the 13th {IFAC} World Congress}, volume~H,
  pages 493--496, 1996.

\bibitem{decockfoga}
J.~Decock and O.~Teytaud.
\newblock Noisy optimization complexity under locality assumption.
\newblock In {\em Proceedings of the twelfth workshop on Foundations of genetic
  algorithms XII}, FOGA XII '13, pages 183--190, New York, NY, USA, 2013. ACM.

\bibitem{dupacEnglish}
V.~Dupa\v{c}.
\newblock Notes on stochastic approximation methods.
\newblock {\em Czechoslovak Mathematical Journal}, 08(1):139--149, 1958.

\bibitem{fabian}
V.~Fabian.
\newblock {Stochastic Approximation of Minima with Improved Asymptotic Speed}.
\newblock {\em {Annals of Mathematical statistics}}, 38:191--200, 1967.

\bibitem{recht}
K.~G. Jamieson, R.~Nowak, and B.~Recht.
\newblock Query complexity of derivative-free optimization.
\newblock In F.~Pereira, C.~Burges, L.~Bottou, and K.~Weinberger, editors, {\em
  Advances in Neural Information Processing Systems 25}, pages 2672--2680.
  Curran Associates, Inc., 2012.

\bibitem{kiefer1952stochastic}
J.~Kiefer, J.~Wolfowitz, et~al.
\newblock Stochastic estimation of the maximum of a regression function.
\newblock {\em The Annals of Mathematical Statistics}, 23(3):462--466, 1952.

\bibitem{rolet2010adaptive}
P.~Rolet and O.~Teytaud.
\newblock Adaptive noisy optimization.
\newblock In C.~Di~Chio, S.~Cagnoni, C.~Cotta, M.~Ebner, A.~Ekárt,
  A.~Esparcia-Alcazar, C.-K. Goh, J.~Merelo, F.~Neri, M.~PreuÃY, J.~Togelius,
  and G.~Yannakakis, editors, {\em Applications of Evolutionary Computation},
  volume 6024 of {\em Lecture Notes in Computer Science}, pages 592--601.
  Springer Berlin Heidelberg, 2010.

\bibitem{shamir}
O.~Shamir.
\newblock On the complexity of bandit and derivative-free stochastic convex
  optimization.
\newblock In {\em {COLT} 2013 - The 26th Annual Conference on Learning Theory,
  June 12-14, 2013, Princeton University, NJ, {USA}}, pages 3--24, 2013.

\bibitem{magiceigen}
M.~Zedek.
\newblock Continuity and location of zeroes of linear combinations of
  polynomials.
\newblock {\em Proc. Amer. Math. Soc.}, 16:78--84, 1965.

\end{thebibliography}
%\bibliography{../../sandraswork/Others/tout,../../mltex/ml,../../sandraswork/Others/tout2,../../sandraswork/Others/tout3,../../sandraswork/Others/teytaud,mca}
\end{document}